\documentclass[review,12pt]{elsarticle}
\usepackage{amsmath}
\usepackage{amsfonts}
\usepackage{geometry}   
\usepackage[parfill]{parskip}    
\usepackage{graphicx}
\usepackage{amssymb}
\usepackage{epstopdf}
\usepackage{psfrag}

\newcommand {\bea}{\begin{eqnarray}}
\newcommand {\ea}{\end{eqnarray}}

\newtheorem{theorem}{Theorem}[section]
\newtheorem{Assumption}{Assumption}[section]
\newtheorem{lemma}{Lemma}[section]
\newtheorem{remark}{Remark}[section]

\newenvironment{proof}[1][Proof]{\textbf{#1.} }{\hspace{\stretch{1}}\rule{0.5em}{0.5em}}

\usepackage{xspace}

\usepackage{subfigure}

\newcommand{\thmref}[1]{{Theorem~\ref{#1}}}
\newcommand{\lemref}[1]{{Lemma~\ref{#1}}}
\newcommand{\secref}[1]{{Section~\ref{#1}}}
\newcommand{\assref}[1]{{Assumption~\ref{#1}}}

\newcommand{\rmref}[1]{{Remark~\ref{#1}}}
\begin{document}
\begin{frontmatter}
\title{Optimal error estimate of the finite element approximation of second order semilinear non-autonomous  parabolic PDEs}


\author[at,atb,atc]{Antoine Tambue}
\cortext[cor1]{Corresponding author}
\ead{antonio@aims.ac.za}
\address[at]{Department of Computing Mathematics and Physics,  Western Norway University of Applied Sciences, Inndalsveien 28, 5063 Bergen.}
\address[atb]{Center for Research in Computational and Applied Mechanics (CERECAM), and Department of Mathematics and Applied Mathematics, University of Cape Town, 7701 Rondebosch, South Africa.}
\address[atc]{The African Institute for Mathematical Sciences(AIMS) of South Africa,
6--8 Melrose Road, Muizenberg 7945, South Africa.}

\author[jdm]{Jean Daniel Mukam}
\ead{jean.d.mukam@aims-senegal.org}
\address[jdm]{Fakult\"{a}t f\"{u}r Mathematik, Technische Universit\"{a}t Chemnitz, 09126 Chemnitz, Germany.}


\begin{abstract}
In this work, we investigate the numerical approximation  of  the second order non-autonomous semilnear parabolic partial differential equation (PDE) using the finite element method.
 To the best of our knowledge, only the linear case is investigated in the literature.  
 Using an approach based on evolution operator depending on two parameters, we obtain the error estimate of the scheme toward the mild solution of the PDE under polynomial growth condition of the nonlinearity.  
 Our convergence rate are obtain for smooth and non-smooth initial data and  is similar to that of  the autonomous case. Our convergence result for smooth initial data is very important in numerical analysis. For instance, it is one step forward in approximating non-autonomous stochastic partial differential equations by the finite element method. In addition, we provide realistic conditions on the nonlinearity, appropriated to achieve optimal convergence rate
   without  logarithmic reduction  by exploiting the smooth properties of the two parameters evolution operator.
\end{abstract}

\end{frontmatter}
\section{Introduction}
\label{intro}
Nonlinear partial differential equations are powerful tools in modelling real-world phenomena in many fields such as in geo-engineering. For instance processes such as oil 
and gas recovery from hydrocarbon reservoirs and mining heat from geothermal reservoirs can be modelled by nonlinear equations with possibly degeneracy 
appearing in the diffusion and transport terms. Since explicit solutions of many PDEs are rarely known, numerical approximations are forceful ingredients to
 quantify them. Approximations are usually done at two levels, namely space and time approximations.
In this paper, we focus on  spatial approximation of the following advection-diffusion problem with a nonlinear reaction term using the finite element method. 
\begin{eqnarray}
\label{model}
\frac{\partial u}{\partial t}=\mathcal{A}(t)u+F(t,u),  \quad u(0)=u_0, \quad t\in(0,T], \quad T>0,
\end{eqnarray}
on the Hilbert  space $H=L^2(\Lambda)$, where $\Lambda$ is an open bounded subset of $\mathbb{R}^d$ $(d=1,2,3)$, with smooth boundary.  The second order differential operator $\mathcal{A}(t)$ is given by 
\begin{eqnarray}
\label{family}
\mathcal{A}(t)u=\sum_{i,j=1}^d\frac{\partial}{\partial x_i}\left(q_{ij}(t,x)\frac{\partial u}{\partial x_j}\right)-\sum_{j=1}^dq_j(t,x)\frac{\partial u}{\partial x_j}+q_0(t,x)u,
\end{eqnarray}
where $q_{i,j}, q_{j}$ and $q_0$ are smooth coefficients. Also, there exists $c_1\geq0$, $0<\gamma\leq 1$ such that 
\begin{eqnarray*}
\vert q_{i,j}(t,x)-q_{i,j}(s,x)\vert\leq c_2\vert t-s\vert^{\gamma},\quad x\in\Lambda,\; t,s\in[0, T],\; i,j\in\{1,\cdots,d\}.
\end{eqnarray*}
Moreover, $q_{i,j}$ satisfies the following ellipticity condition
\begin{eqnarray}
\label{ellip}
\sum_{i,j=1}^dq_{ij}(t,x)\xi_i\xi_j\geq c\vert \xi\vert^2, \quad (t,x)\in [0,T]\times \overline{\Lambda}, 
\end{eqnarray}
where $c> 0$ is a constant. 
The finite element approximation of \eqref{model} with constant linear operator $\mathcal{A}(t)=\mathcal{A}$ are widely investigated in the scientific 
literature, see e.g. \cite{Stig2,Thomee2,Suzuki,Antjd2} and the references therein. The finite volume method for $\mathcal{A}(t)=\mathcal{A}$ was recently investigated in \cite{Tambueseul}. If we turn our attention 
to the  non-autonomous case, the list of references becomes remarkably short. In the linear homogeneous case ($F(t,u)=0$), the finite element approximation  has been 
investigated  in \cite{Luskin}, \cite[Chapter III, Section 14.2]{Suzuki}.
 The linear inhomogeneous version of \eqref{model} ($F(t,u)=f(t)$) 
 was investigated  in \cite{Luskin,Dimitri,Thomee1}, \cite[Chapter III, Section 12]{Suzuki} and the references therein. To the best of our knowledge, the nonlinear case is not yet investigated in the scientific literature. 
  This paper fills that gap by  investigating the error estimate of the finite  element method of \eqref{model} with a  nonlinear source $F(t,u)$, which is  more challenging due to the presence of the unknown $u$ in the source term $F$. This become more challenging when the nonlinear function satisfies the polynomial growth condition. Our strategy is based on an introduction of two parameters evolution operator  by exploiting carefully its smooth regularity properties.  Our key intermediate result, namely \lemref{spaceerrorlemma} generalizes \cite[Theorem 3.5]{Thomee2} for time dependent and not necessary self-adjoint operators.  It also generalizes \cite[Theorem 4.2]{Thomee2}, the results in \cite[Chapter III, Section 12]{Suzuki} and in \cite{Luskin,Dimitri,Thomee1} to smooth and non-smooth initial data. Note that \lemref{spaceerrorlemma} for non-smooth initial data is of great important in numerical analysis. It is key to obtain the  convergence of the finite element  method for many nonlinear problems, including stochastic partial differential equations(SPDEs), see e.g. \cite{Kovcas,Kruse1,Xiaojie2} and references therein for time independent SPDEs. In fact, in the case of SPDEs, due to the It\^{o}-isometry formula or the Burkh\"{o}lder Davis-Gundy inequality,  the non-smooth version of \lemref{spaceerrorlemma} cannot be applied since it brings degenerates integrals, which causes difficulties in  the error  estimates or reduces considerably the order of convergence. Hence our result is more general than the existing results and also has many applications. 
  The convergence rate achieved for semilinear problem is in agreement with many results in the literature on autonomous problems and on  non-autonomous linear problems. More precisely, we achieve convergence order   $\mathcal{O}\left(h^{2}t^{-1+\beta/2}+h^2\left(1+\ln(t/h^2)\right)\right)$ or $\mathcal{O}(h^{\beta})$, where $\beta$ is a regularity parameter defined in \assref{assumption1}. Under optimal regularity of the nonlinear function $F$ or under a linear growth assumption on $F$, we achieve optimal convergence order $\mathcal{O}(h^2t^{-1+\beta/2})$. Following \cite{Tambueseul}  and using  the similar approach based on the two parameters evolution operator, this work can be extended to the finite volume method.
 The rest of this paper is structured as follows.  In  Section \ref{nummethod},  the well-posedness results  are provided along with the finite element approximation. The error estimate    is analysed  in Section \ref{proof1} for both Lipschitz nonlinearity and polynomial growth nonlinearity. 
 
\section{Mathematical setting and numerical method}
 \label{nummethod}
 \subsection{Notations, settings and well well-posedness problem}
We denote by $\Vert \cdot \Vert$ the norm associated to
the inner product $\langle\cdot ,\cdot \rangle_H$ in the Hilbert space $H=L^{2}(\Lambda)$.  We denote by $\mathcal{L}(H)$  the set of bounded linear operators  
in $H$. Let $\mathcal{C}:=\mathcal{C}(\overline{\Lambda}, \mathbb{R})$ be the set of continuous functions equipped with the norm $\Vert u\Vert_{\mathcal{C}}=\sup\limits_{x\in\overline{\Lambda}}\vert u(x)\vert$, $u\in\mathcal{C}$. Next, we make the following  assumptions.
 \begin{Assumption}
 \label{assumption1}
The initial data $u_0$   belongs to $\mathcal{D}\left(\left(-A(0)\right)^{\frac{\beta}{2}}\right)$, $0\leq \beta\leq 2$.
 \end{Assumption}
 \begin{Assumption}
 \label{assumption3}
 The nonlinear function $F : [0,T]\times H\longrightarrow H$  is Lipschitz continuous, i.e. there exists a constant $K$ such that
 \begin{eqnarray}
 \label{Lipschitz}
 \Vert F(t,v)-F(s,w)\Vert\leq K(\vert t-s\vert+\Vert v-w\Vert),\quad s,t\in[0, T],\quad v,w\in H.
 \end{eqnarray}
 \end{Assumption}
We introduce two spaces $\mathbb{H}$ and $V$, such that $\mathbb{H}\subset V$,  depending on the boundary conditions of $- \mathcal{A}(t)$. For  Dirichlet  boundary conditions, we take $
V=\mathbb{H}=H^1_0(\Lambda)$. 
For Robin  boundary condition,  we take $V=H^1(\Lambda)$ and
\begin{eqnarray}
\label{espaceR}
\mathbb{H}=\{v\in H^2(\Lambda) : \partial v/\partial v_{\mathcal{A}}+\alpha_0v=0,\quad \text{on}\quad \partial \Lambda\}, \quad \alpha_0\in\mathbb{R},
\end{eqnarray}
where $\partial v/\partial v_{\mathcal{A}}$ stands for the differentiation along the outer conormal vector $v_{\mathcal{A}}$.
One can easily check that \cite[Chapter III, (11.14$^{\prime}$)]{Suzuki} the bilinear operator $a(t)$, associated to $-\mathcal{A}(t)$  defined by $a(t)(u,v)=\langle-\mathcal{A}(t)u,v\rangle_H,\quad u\in\mathcal{D}(\mathcal{A}(t)),\quad v\in V$ satisfies 
\begin{eqnarray}
\label{ellip2}
a(t)(v,v)\geq \; \lambda_0\Vert v\Vert_{1}^{2},\;\;\;\;\; v \in V,\quad t\in[0,T],
\end{eqnarray}
where $\lambda_0$ is a positive constant,  independent of $t$.
Note that $a(t)(\cdot,\cdot)$ is bounded in $V\times V$ (\cite[Chapter III, (11.13)]{Suzuki}), so the following operator $A(t):V \rightarrow V^*$ is well defined
\begin{eqnarray*}
a(t)(u,v) = \langle -A(t) u, v \rangle \quad u, v\in V,\quad t\in[0, T],
\end{eqnarray*}
where $V^*$ is the dual space of V and $\langle\cdot ,\cdot \rangle$ the duality pairing between $V^*$ and $V$.
Identifying $H$ to its adjoint space $H^*$,  we get the following continuous and dense inclusions
\begin{eqnarray*}
V \subset  H \subset V^*,\quad \text{and therefore}\quad \langle u, v \rangle_H=\langle u, v \rangle,   \quad u \in H,\quad  v\in V.
\end{eqnarray*}
 So  if we want to replace  $\langle\cdot ,\cdot \rangle$   by  the scalar product  of   $\langle\cdot,\cdot \rangle_H$ on $H$, we  therefore need to have
 $A(t) u \in H$,  for $u \in V$. So  the domain  of  $-A(t)$  is defined as $$D:= \mathcal{D}\left(-A(t)\right) =\mathcal{D}\left(A(t)\right) =\{ u\in V,\,\,  A(t)u \in H \}.$$
 It is well known that  \cite[Chapter III, (11.11) \& (11.11$^{\prime}$)]{Suzuki} in the case of Dirichlet boundary conditions $D=H^1_0(\Lambda)\cap H^2(\Lambda)$ and in the case of Robin  boundary conditions $D=\mathbb{H}$ in \eqref{espaceR}. 
 We write the restriction  of  $A(t):V \longrightarrow V^*$ to  $\mathcal{D}\left(A(t)\right)$  again   $A(t)$ which is therefore regarded as
an operator of $H$ (more precisely the $H$ realization of  $\mathcal {A}(t)$).

The coercivity  property \eqref{ellip2}  implies that $-A(t)$ is a positive operator and its fractional powers are well defined (\cite{Stig2,Suzuki}).  
 The following equivalence of norms holds \cite{Suzuki,Stig2} 
\begin{eqnarray}
\label{equivalence}
\Vert v\Vert_{H^{\alpha}(\Lambda)}&\equiv& \Vert ((-A(t))^{\frac{\alpha}{2}}v\Vert:=\Vert v\Vert_{\alpha},\; v\in \mathcal{D}((-A(t))^{\frac{\alpha}{2}})\cap H^{\alpha}(\Lambda),\quad \alpha\in[0, 2].
\end{eqnarray}
It is well known that the family of operators $\{A(t)\}_{0\leq t\leq T}$ generate a two parameters operators $\{U(t,s)\}_{0\leq s\leq t\leq T}$, see e.g. \cite{Sobolev} or \cite[Page 832]{Suzuki}. 
The evolution equation \eqref{model} can be written as follows
\begin{eqnarray}
\label{semi0}
\dfrac{du(t)}{dt}=A(t)u(t)+F(t,u(t)), \quad u(0)=u_0,\quad t\in(0,T].
\end{eqnarray}
The following theorem provides the well posedness of problem \eqref{model} (or \eqref{semi0}).
\begin{theorem}\cite{Sobolev}
\label{theorem1}
Let \assref{assumption3} be fulfilled. If $u_0\in H$, then the initial value problem \eqref{model} has a unique mild solution $u(t)$ given by
\begin{eqnarray}
\label{mild0}
u(t)=U(t,0)u_0+\int_0^tU(t,s)F(s,u(s))ds,\quad t\in(0,T].
\end{eqnarray}
 Moreover, if \assref{assumption1} is fulfilled, then the following space regularity holds \footnote{This estimate also holds when $u$ is replaced by its semi-discrete version $u^h$ defined in \secref{finiteelement}.}
\begin{eqnarray}
\label{spacereg1}
\Vert (-A(s))^{\frac{\beta}{2}}u(t)\Vert+\Vert F(u(t))\Vert\leq C\left(1+\Vert (-A(s))^{\frac{\beta}{2}}u_0\Vert\right),\quad \beta\in[0,2),\quad s,t\in[0,T].
\end{eqnarray}
\end{theorem}

\subsection{Finite element discretization}
\label{finiteelement}
Let $\mathcal{T}_h$ be a triangulation of $\Lambda$ with maximal length $h$. Let $V_h \subset V$ denotes the space of continuous and piecewise 
linear functions over the triangulation $\mathcal{T}_h$. 
 We  defined  the projection $P_h$  from  $H=L^2(\Lambda)$ to $V_h$ by 
\begin{eqnarray}
\label{discrete1}
\left\langle P_hu,\chi\right\rangle_H=\langle u,\chi\rangle_H, \quad  \chi\in V_h,\,  u\in H.
\end{eqnarray}
For any $t\in[0, T]$, the discrete operator $A_h(t) : V_h\longrightarrow V_h$ is defined by 
\begin{eqnarray}
\label{discrete2}
\left\langle A_h(t)\phi,\chi\right\rangle_H=\left\langle A(t)\phi,\chi\right\rangle_H=-a(t)(\phi,\chi),\quad  \phi \in D \cap V_h, \chi\in V_h.
\end{eqnarray}
The space semi-discrete version of problem \eqref{semi0} consists of  finding $u^h(t)\in V_h$ such that 
\begin{eqnarray}
\label{semi1}
\dfrac{du^h(t)}{dt}=A_h(t)u^h(t)+P_hF(t,u^h(t)), \quad u^h(0)=P_hu_0,\quad t\in(0,T].
\end{eqnarray}
For  $t\in[0,T]$,  we  introduce the Ritz projection $R_h(t) :V\longrightarrow V_h$ defined by 
\begin{eqnarray}
\label{ritz1}
\left\langle -A(t)R_h(t)v,\chi\right\rangle_H=\left\langle -A(t)v,\chi\right\rangle_H=a(t)(v,\chi),\quad v\in V \cap D,\quad \chi\in V_h.
\end{eqnarray}
It is well known (see e.g.  \cite[(3.2)]{Luskin} or \cite{Suzuki}) that the following error estimate holds
\begin{eqnarray}
\label{ritz2}
\Vert R_h(t)v-v\Vert+h\Vert R_h(t)v-v\Vert_{H^1(\Lambda)}\leq Ch^{r}\Vert v\Vert_{H^{r}(\Lambda)},\quad v\in V\cap H^{r}(\Lambda),\quad r\in[1,2].
\end{eqnarray}
The following error estimate also holds (see e.g. \cite[(3.3)]{Luskin} or \cite{Suzuki}) 
\begin{eqnarray}
\label{ritz3}
\Vert D_t\left(R_h(t)v-v\right)\Vert+h\Vert D_t\left(R_h(t)v-v\right)\Vert_{H^1(\Lambda)}\leq Ch^{r}\left(\Vert v\Vert_{H^{r}(\Lambda)}+\Vert D_tv\Vert_{H^{r}(\Lambda)}\right),
\end{eqnarray}
for any $r\in[1,2]$ and $v\in V\cap H^r(\Lambda)$, where $D_t:=\frac{\partial }{\partial t}$ and $D_tR_h(t)=R^{\prime}_h(t)$ is the time derivative of $R_h$. 
According to the generation theory, $A_h(t)$ generates a two parameters evolution operator $\{U_h(t,s)\}_{0\leq s\leq t\leq T}$, see e.g. \cite[Page 839]{Suzuki}.
Therefore the  mild solution of \eqref{semi1} can be written as follows
\begin{eqnarray}
\label{mild4}
u^h(t)=U_h(t,0)P_hu_0+\int_0^tU_h(t,s)P_hF(s,u^h(s))ds,\quad t\in[0,T].
\end{eqnarray}
In the rest of this paper, $C\geq 0$ stands for a constant indepemdent of $h$, that may change from one place to another. It is well known (see e.g. \cite[Chapter III, (12.3) \& (12.4)]{Suzuki}) that for any $0\leq\gamma\leq\alpha\leq 1$ and $0\leq s< t\leq T$, the following estimates hold\footnote{These estimates remain true if $A_h(t)$ and $U_h(t,s)$ are replaced by $A(t)$ and $U(t,s)$ respectively.}
\begin{eqnarray}
\label{ae1}
\left\Vert (-A_h(t))^{\alpha}U_h(t,s)\right\Vert_{\mathcal{L}(H)}\leq C(t-s)^{-\alpha},\quad
\left\Vert U_h(t,s)(-A_h(s))^{\alpha}\right\Vert_{\mathcal{L}(H)}\leq C(t-s)^{-\alpha}.
\end{eqnarray}

\section{Main result}
\label{proof1}
\subsection{Preliminaries result}
We consider the following linear homogeneous problem: find $w\in D\subset V$ such that
\begin{eqnarray}
\label{determ1}
w'=A(t)w,\quad w(\tau)=v,\quad t\in(\tau,T],\quad \text{with}\quad 0\leq\tau\leq T.
\end{eqnarray}
The corresponding semi-discrete problem in space is: find $w_h\in V_h$ such that
\begin{eqnarray}
\label{determ2}
w_h'(t)=A_h(t)w_h,\quad w_h(\tau)=P_hv,\quad t\in(\tau,T],\quad \text{with}\quad 0\leq\tau\leq T.
\end{eqnarray}
The following lemma will be useful in our convergence analysis. 
\begin{lemma}
\label{spaceerrorlemma}
Let $r\in[0,2]$ and  $\gamma\leq r$. Let  \assref{assumption3} be fulfilled.  Then the following error estimate holds for the semi-discrete approximation  \eqref{determ2}
\begin{eqnarray*}
\label{er0}
\left\Vert w(t)-w_h(t)\right\Vert=\left\Vert[U(t,\tau)-U_h(t,\tau)P_h]v\right\Vert\leq Ch^r(t-\tau)^{-\frac{(r-\gamma)}{2}}\Vert v\Vert_{\gamma},\; v\in \mathcal{D}\left(\left(-A(0)\right)^{\frac{\gamma}{2}}\right).
\end{eqnarray*}
\end{lemma}
\begin{proof}
We split the desired error as follows
\begin{eqnarray}
\label{espa0}
w_h(t)-w(t)=\left(w_h(t)-R_h(t)w(t)\right)+\left(R_h(t)w(t)-w(t)\right)\equiv \theta(t)+\rho(t).
\end{eqnarray}
Using the definition of $R_h(t)$ and $P_h$ (\eqref{discrete1}--\eqref{discrete2}), we can prove exactly as in \cite{Stig2} that
\begin{eqnarray}
\label{espacetamb}
A_h(t)R_h(t)=P_hA(t),\quad t\in[0,T].
\end{eqnarray}
One can easily compute the following derivatives 
\begin{eqnarray}
\label{espace1a}
D_t\theta&=&A_h(t)w_h(t)-D_tR_h(t)w(t)-R_h(t)A(t)w(t),\\
\label{espace1b}
D_t\rho&=&D_tR_h(t)w(t)+R_h(t)A(t)w(t)-A(t)w(t).
\end{eqnarray}
Endowing $V$ and the linear subspace $V_h$ with the norm $\Vert .\Vert_{H^1(\Lambda)}$, it follows from \eqref{ritz2} that $R_h(t)\in L(V, V_h)$,  $t\in [0, T]$. By the definition of the differential operator, it follows that $D_tR_h(t)\in L(V, V_h)$ for all $t\in[0, T]$. Hence $P_hD_tR_h(t)=D_tR_h(t)$ for all $t\in[0,T]$ 
and it follows from \eqref{espace1b} that 
\begin{eqnarray}
\label{espace1c}
P_hD_t\rho=D_tR_h(t)w(t)+R_h(t)A(t)w(t)-P_hA(t)w(t).
\end{eqnarray}
Adding and subtracting $P_hA(t)w(t)$ in \eqref{espace1a} and using \eqref{espacetamb}, it follows that
\begin{eqnarray}
\label{espa1}
D_t\theta=A_h(t)\theta-P_hD_t\rho,\quad t\in(\tau,T],
\end{eqnarray}
From \eqref{espace1a}, the mild solution of $\theta$ is given by
\begin{eqnarray}
\label{espa2}
\theta(t)=U_h(t,\tau)\theta(\tau)-\int_{\tau}^tU_h(t,s)P_hD_s\rho(s)ds.
\end{eqnarray}
Splitting the integral part of \eqref{espa2} in two  and integrating by parts  the first one yields
\begin{eqnarray}
\label{espa3}
\theta(t)&=& U_h(t,\tau)\theta(\tau)+U_h(t,\tau)P_h\rho(\tau)-U_h\left(t,(t+\tau)/2\right)P_h\rho\left((t+\tau)/2\right)\nonumber\\
&+&\int_{\tau}^{(t+\tau)/2}\frac{\partial}{\partial s}\left(U_h(t,s)\right)P_h\rho(s)ds-\int_{(t+\tau)/2}^tU_h(t,s)P_hD_s\rho(s)ds.
\end{eqnarray}
Using the expression of $\theta(\tau)$, $\rho(\tau)$ (see \eqref{espa0}) and the fact that $u_h(\tau)=P_hv$, it holds that $\theta(\tau)+P_h\rho(\tau)=0$.
Hence  \eqref{espa3} reduces  to
\begin{eqnarray}
\label{espa5}
\theta(t)=-U_h(t,s)P_h\rho((t+\tau)/2)+\int_{\tau}^{\frac{(t+\tau)}{2}}\frac{\partial}{\partial s}\left(U_h(t,s)
\right)P_h\rho(s)ds-\int_{\frac{(t+\tau)}{2}}^tU_h(t,s)P_hD_s\rho(s)ds.
\end{eqnarray}
Taking the norm in both sides of \eqref{espa5} and using \eqref{ae1} yields
\begin{eqnarray}
\label{espa6}
\Vert\theta(t)\Vert&\leq& C\left\Vert \rho\left((t+\tau)/2\right)\right\Vert+\int_{\tau}^{\frac{(t+\tau)}{2}}\left\Vert U_h(t,s)A_h(s)\right\Vert_{\mathcal{L}(H)}\Vert \rho(s)\Vert ds+\int_{\frac{(t+\tau)}{2}}^t\Vert D_s\rho(s)\Vert ds\nonumber\\
&\leq& C\left\Vert \rho\left((t+\tau)/2\right)\right\Vert+\int_{\tau}^{\frac{(t+\tau)}{2}}(t-s)^{-1}\Vert \rho(s)\Vert ds+\int_{\frac{(t+\tau)}{2}}^t\Vert D_s\rho(s)\Vert ds.
\end{eqnarray}
Using \eqref{ritz2} and \eqref{ritz3}, it holds that
\begin{eqnarray}
\label{espa7}
\Vert \rho(s)\Vert\leq  Ch^r\Vert w(s)\Vert_r,\quad \Vert D_s\rho(s)\Vert \leq Ch^r\left(\Vert w(s)\Vert_r+\Vert D_sw(s)\Vert_r\right).
\end{eqnarray}
Note that the solution of \eqref{determ1} can be represented as follows.
\begin{eqnarray}
\label{encore1}
w(s)=U(s,\tau)v,\quad s\geq \tau.
\end{eqnarray}
Pre-multiplying both sides of \eqref{encore1} by $(-A(s))^{\frac{r}{2}}$ and using \eqref{ae1} yields
\begin{eqnarray}
\label{encore2}
\left\Vert (-A(s))^{\frac{r}{2}}w(s)\right\Vert&\leq& \left\Vert (-A(s))^{\frac{r}{2}}U(s,\tau)(-A(\tau))^{-\frac{\gamma}{2}}\right\Vert_{\mathcal{L}(H)}\left\Vert (-A(\tau))^{\frac{\gamma}{2}}v\right\Vert\nonumber\\
&\leq& C(s-\tau)^{-\frac{(r-\gamma)}{2}}\left\Vert (-A(\tau))^{\frac{\gamma}{2}}v\right\Vert\leq C(s-\tau)^{-\frac{(r-\gamma)}{2}}\Vert v\Vert_{\gamma}.
\end{eqnarray}
Therefore it holds that
\begin{eqnarray}
\label{espa8}
\Vert w(s)\Vert_r\leq C(s-\tau)^{-\frac{(r-\gamma)}{2}}\Vert v\Vert_{\gamma}, \quad 0\leq \gamma\leq r\leq 2,\quad \tau<s.
\end{eqnarray}
Substituting \eqref{espa8} in \eqref{espa7} yields
\begin{eqnarray}
\label{espa8a}
\Vert \rho(s)\Vert_r\leq Ch^r(s-\tau)^{-\frac{(r-\gamma)}{2}}\Vert v\Vert_{\gamma}.
\end{eqnarray}
Taking the derivative with respect to $s$ in both sides of \eqref{encore1} yields
\begin{eqnarray}
\label{encore3}
D_sw(s)=-A(s)U(s,\tau)v.
\end{eqnarray}
As for  \eqref{encore2}, pre-multiplying both sides of \eqref{encore3} by $(-A(s))^{\frac{r}{2}}$ and using \eqref{ae1}  yields
\begin{eqnarray}
\label{encore4a}
\Vert D_sw(s)\Vert_r\leq C(s-\tau)^{-1-\frac{(r-\gamma)}{2}}\Vert v\Vert_{\gamma}.
\end{eqnarray}
Substituting \eqref{espa8} and  \eqref{encore4a} in the second estimate of \eqref{espa7} yields
\begin{eqnarray}
\label{espa15}
\Vert D_s\rho(s)\Vert\leq Ch^r\left((s-\tau)^{-\frac{(r-\gamma)}{2}}\Vert v\Vert_{\gamma}+(s-\tau)^{-1-\frac{(r-\gamma)}{2}}\Vert v\Vert_{\gamma}\right)\leq Ch^r(s-\tau)^{-1-\frac{(r-\gamma)}{2}}\Vert v\Vert_{\gamma}.
\end{eqnarray}
Substituting the first estimate of \eqref{espa7} and \eqref{espa15} in \eqref{espa6} and using \eqref{espa8a}  yields 
\begin{eqnarray}
\label{espa17}
\Vert\theta(t)\Vert&\leq& Ch^r(t-\tau)^{-\frac{(r-\gamma)}{2}}\Vert v\Vert_{\gamma}+Ch^r\int_{\tau}^{\frac{t+\tau}{2}}(t-s)^{-1}(s-\tau)^{-\frac{(r-\gamma)}{2}}\Vert v\Vert_{\gamma}ds\nonumber\\
&+&Ch^r\int_{\frac{t+\tau}{2}}^t(s-\tau)^{-1-\frac{(r-\gamma)}{2}}\Vert v\Vert_{\gamma}ds.
\end{eqnarray}
Using the estimate 
\begin{eqnarray}
\label{espa18}
\int_{\tau}^{\frac{t+\tau}{2}}(t-s)^{-1}(s-\tau)^{-\frac{(r-\gamma)}{2}}ds+\int_{\frac{t+\tau}{2}}^t(s-\tau)^{-1-\frac{(r-\gamma)}{2}}ds\leq C(t-\tau)^{-\frac{(r-\gamma)}{2}},\nonumber
\end{eqnarray}
it follows  from \eqref{espa17} that
\begin{eqnarray}
\label{espa20}
\Vert \theta(t)\Vert\leq  Ch^r(t-\tau)^{-\frac{(r-\gamma)}{2}}\Vert v\Vert_{\gamma}.
\end{eqnarray}
Substituting \eqref{espa20} and \eqref{espa8a} in \eqref{espa0}  completes the proof of \lemref{spaceerrorlemma}.
\end{proof}

\subsection{Error estimate of the semilinear problem under global Lipschitz condition}
\begin{theorem}
\label{theorem2}
Let  Assumptions  \ref{assumption1} and \ref{assumption3}  be fulfilled. Let $u(t)$ and $u^h(t)$ be defined by \eqref{mild0} and \eqref{mild4} respectively. 
Then the following error estimate holds 
\begin{eqnarray}
\label{time1}
\Vert u(t)-u^h(t)\Vert\leq Ch^{2}t^{-1+\beta/2}+Ch^2(1+\ln(t/h^2)),\quad 0< t\leq T.
\end{eqnarray}
If in addition the nonlinearity $F$ satisfies the linear growth condition $\Vert F(t,v)\Vert\leq C\Vert v\Vert$ or if there exists $\delta>0$ small enough such that $\Vert(-A(s))^{\delta}F(t, v)\Vert\leq Ct+C\Vert (-A(s))^{\delta}v\Vert$, $s,t\in[0, T]$, $v\in H$, then the following optimal error estimate holds
\begin{eqnarray}
\label{time2}
\Vert u(t)-u^h(t)\Vert\leq Ch^{2}t^{-1+\beta/2},\quad 0<t\leq T,
\end{eqnarray}
where $\beta$ is defined in Assumption \ref{assumption1}.
\end{theorem}
\begin{remark}
Note that the hypothesis $\Vert F(t,v)\Vert \leq C\Vert v\Vert$ is  not too restrictive. An example of class of nonlinearities for which such hypothesis is fulfilled is a class of functions satisfying $F(t, 0)=0$, $t\in[0, T]$. Concrete examples are  operators of the form $F(t,v)=f(t)\frac{v}{1+\vert v\vert}$, with $f:[0, T]\longrightarrow\mathbb{R}$ continuous or bounded on $[0, T]$. 
\end{remark}
\begin{remark}
It is possible to obtain an error estimate without irregularities terms of the form $t^{-1+\beta/2}$ with a drawback that the convergence rate will not be $2$, but will depend on the regularity of the initial data. The proof follows the same lines as that of \thmref{theorem2} using \lemref{spaceerrorlemma} and this yields
\begin{eqnarray*}
\Vert u(t)-u^h(t)\Vert\leq Ch^{\beta},\quad t\in [0, T]. 
\end{eqnarray*}
\end{remark}

\begin{proof}  of \thmref{theorem2}. 
 We start with  the proof of \eqref{time1}.
Subtracting \eqref{mild4} form \eqref{mild0},  taking the  norm in both sides and using triangle inequality yields
\begin{eqnarray}
\label{estiI}
\Vert u(t)-u^h(t)\Vert&\leq& \left\Vert U(t,0)u_0-U_h(t,0)P_hu_0\right\Vert\nonumber\\
&+&\left\Vert\int_0^{t}\left[U(t,s)F\left(s,u(s)\right)-U_h(t,s)P_hF\left(s,u^h(s)\right)\right]
 ds\right\Vert=:I_0+I_1.
\end{eqnarray}
Using \lemref{spaceerrorlemma} with $r=2$ and $\gamma=\beta$  yields
\begin{eqnarray}
\label{jour1}
I_0\leq Ch^{2}t^{-1+\beta/2}\Vert u_0\Vert_{\beta}\leq Ch^{2}t^{-1+\beta/2}.
\end{eqnarray}
Using  \assref{assumption3}, \eqref{ae1} and \eqref{spacereg1} yields
\begin{eqnarray}
\label{estiI2}
I_1&\leq& \int_0^t\left\Vert U(t,s)\left[F\left(s,u(s)\right)-F\left(s,u^h(s)\right)\right]\right\Vert ds+\int_0^t\left\Vert \left[U(t,s)-U_h(t,s)P_h\right]F\left(s,u^h(s)\right)\right\Vert ds\nonumber\\
&\leq& C\int_0^t\left\Vert u(s)-u^h(s)\right\Vert ds+C\int_0^t\left\Vert \left[U(t,s)-U_h(t,s)P_h\right]]F\left(s,u^h(s)\right)\right\Vert ds.
\end{eqnarray}
If $0\leq t\leq h^2$, then using \eqref{ae1} easily yields   $I_1\leq Ch^2+\int_0^t\Vert u(s)-u^h(s)\Vert ds$. If $0<h^2\leq t$, using  \lemref{spaceerrorlemma} (with $r=2$ and $\gamma=0$), and splitting the second integral in two parts yields
\begin{eqnarray}
\label{estiI3}
I_1&\leq& C\int_0^t\Vert u(s)-u^h(s)\Vert ds+Ch^2\int_0^{t-h^2}(t-s)^{-1}ds+Ch^2\int_{t-h^2}^t(t-s)^{-1}ds\nonumber\\
&\leq& C\int_0^t\Vert u(s)-u^h(s)\Vert ds+Ch^2(1+\ln(t/h^2)).
\end{eqnarray}
Substituting \eqref{estiI3} and \eqref{jour1} in \eqref{estiI} and applying  Gronwall's lemma  proves \eqref{time1}. To prove \eqref{time2}, we only need to re-estimate the term $I_3:=\int_0^t\left\Vert \left[U(t,s)-U_h(t,s)P_h\right]F\left(s,u^h(s)\right)\right\Vert ds$. Note that under assumption $\Vert(-A(s))^{\delta}F(t, v)\Vert\leq Ct+C\Vert (-A(s))^{\delta}v\Vert$, using \lemref{spaceerrorlemma} (with $r=2$ and $\gamma=\delta$) and \eqref{spacereg1}, following the same lines as above one easily obtain $I_3\leq Ch^2$. Let us now estimate $I_3$ under the hypothesis $\Vert F(t,v)\Vert \leq C\Vert v\Vert$. Using \assref{assumption3},  \eqref{spacereg1} and exploiting the mild solution \eqref{mild4} one easily obtain
\begin{eqnarray}
\label{estiI4}
\Vert F(t,u^h(t))\Vert\leq \Vert u^h(t)\Vert\leq C\vert t-s\vert^{\epsilon}s^{-\epsilon},\;\;\; \Vert F(s, u^h(s))-F(t, u^h(t))\Vert\leq C\vert t-s\vert^{\epsilon}s^{-\epsilon},
\end{eqnarray}
for some $\epsilon\in(0, 1)$ and any $s, t\in[0, T]$. Using \lemref{spaceerrorlemma} (with $r=2$ and $\gamma=0$), triangle inequality and \eqref{estiI4} yields
\begin{eqnarray*}
I_3&\leq& Ch^2\int_0^t(t-s)^{-1}\left\Vert F\left(s,u^h(s)\right)-F(t,u^h(t))\right\Vert ds+Ch^2\int_0^t(t-s)^{-1}\Vert F(t,u^h(t))\Vert ds\nonumber\\
&\leq& Ch^2\int_0^t(t-s)^{-1+\epsilon}s^{-\epsilon}ds\leq Ch^2.
\end{eqnarray*}
Hence the new estimate of $I_1$ is given below
\begin{eqnarray}
\label{estiI5}
I_1\leq Ch^2+C\int_0^t\Vert u(s)-u^h(s)\Vert ds.
\end{eqnarray}
Substituting \eqref{estiI5} and \eqref{jour1} in \eqref{estiI} and applying  Gronwall's lemma  proves \eqref{time2} and the proof of \thmref{theorem2} is completed.
\end{proof}

\subsection{Error estimate of the semilinear problem under polynomial growth condition}
\label{sectpoly}
In this section, we take $\beta\in\left(\frac{d}{2}, 2\right]$. We make the following assumptions on the nonlinearity.
\begin{Assumption}
\label{assumption2}
there exist two constants and  $L_1, c_1\in[0, \infty)$ such that the nonlinear function $F$  satisfies the following 
\begin{eqnarray}
\label{Polynome1}
\Vert F(w)\Vert&\leq& L_1+ L_1\Vert w\Vert\left(1+\Vert w\Vert^{c_1}_{\mathcal{C}}\right), \quad w\in H,\\
\label{Polynome2}
 \Vert F(w)-F(v)\Vert&\leq& L_1\Vert w-v\Vert\left(1+\Vert u\Vert^{c_1}_{\mathcal{C}}+\Vert v\Vert^{c_1}_{\mathcal{C}}\right),\quad w, v\in H.
\end{eqnarray}
\end{Assumption}
Let us recall the following Sobolev embedding (continuous embedding). \begin{eqnarray}
\label{sobolev1}
\mathcal{D}\left((-A(0))^{\delta}\right)\subset C\left(\Lambda, \mathbb{R}\right),\quad \text{for}\quad \delta >\frac{d}{2},\quad d\in\{1,2,3\}.
\end{eqnarray}
It is a classical solution that under \assref{assumption2} \eqref{semi0} has a unique mild solution $u$ satisfying\footnote{This remains  true if $u$ is replaced by its discrete version $u^h$.} $u\in C\left([0, T], \mathcal{D}\left((-A(0))^{\beta}\right)\right)$,  see e.g. \cite{Sobolev}. Hence using the Sobolev embbeding \eqref{sobolev1}, it holds that
\begin{eqnarray}
\label{sobolev2}
\Vert u(t)\Vert_{\mathcal{C}}\leq C\left\Vert (-A(0))^{\frac{\beta}{2}}u(t)\right\Vert\leq C,\quad \Vert u^h(t)\Vert_{\mathcal{C}}\leq C\left\Vert (-A(0))^{\frac{\beta}{2}}u^h(t)\right\Vert\leq C,\quad t\in[0, T].
\end{eqnarray}

\begin{theorem}
\label{mainresul2}
Let $u(t)$ and $u^h(t)$ be solution of \eqref{semi0} and \eqref{semi1} respectively. Let Assumptions \ref{assumption1} and \ref{assumption2}  be fulfilled. Then the following error estimate   holds
\begin{eqnarray}
\Vert u(t)-u^h(t)\Vert\leq Ch^{2}t^{-1+\beta/2}+Ch^2\left(1+\ln(t/h^2)\right),\quad t\in[0, T].
\end{eqnarray}
If in addition there exists $c_1, c_2\geq 0$ such that the nonlinearity $F$ satisfies the polynomial growth condition 
\begin{eqnarray}
\label{Polysharp1}
\Vert F(t,v)\Vert\leq C\Vert v\Vert^{c_1}\Vert v\Vert_{\mathcal{C}}^{c_2},
\end{eqnarray}
then the following optimal  error estimate holds
\begin{eqnarray}
\label{time2}
\Vert u(t)-u^h(t)\Vert\leq Ch^{2}t^{-1+\beta/2},\quad 0<t\leq T,
\end{eqnarray}
\end{theorem}
\begin{proof}
The proof goes along the same lines as that of \thmref{theorem2} by using appropriately \assref{assumption2} and \eqref{sobolev2}. 
\end{proof}

\begin{remark}
It is possible in \thmref{mainresul2} to obtain convergence estimate without irregularities terms $t^{-1+\beta/2}$. But the convergence rate will depend on the regularity of the initial data and will be of the form 
\begin{eqnarray*}
\Vert u(t)-u^h(t)\Vert\leq Ch^{\beta},\quad t\in[0, T].
\end{eqnarray*}
\end{remark}

\begin{remark}
 \assref{assumption2} is weaker than \assref{assumption3} and therefore include more nonlinearities. However, the price to pay when using \assref{assumption2} is that one requires more regularity on the initial data. 
\end{remark}

\begin{remark}
Let $\varphi: \mathbb{R}\longrightarrow \mathbb{R}$ be polynomial of any order. The nonlinear operator $F$ is defined as the Nemytskii operator
 \begin{eqnarray*}
 F\left(u\right)(x)=\varphi\left(u(x)\right),\quad u\in H,\quad x\in \Lambda,
 \end{eqnarray*}
 is a concrete example satisfying \assref{assumption2}.
 
 In fact, let us assume without loss of generality that $\varphi$ is polynomial of degree $l>1$, that is
\begin{eqnarray}
\varphi(x)=\sum_{i=0}^la_ix^i,\quad x\in\mathbb{R}.
\end{eqnarray}
\label{Remarkpol}
Note that the proofs  in the cases $l=0, 1$ are obvious. 
For any $u\in H\cap \mathcal{C}(\overline{\Lambda}, \mathbb{R})$, using traingle inequality and the fact $\left(\sum\limits_{i=0}^lc_i\right)^2\leq (l+1)\sum\limits_{i=0}^lc_i^2$, $c_i\geq 0$, we obtain
\begin{eqnarray}
\Vert F(v)\Vert^2&=&\int_{\Lambda}\vert F(v)(x)\vert^2dx=\int_{\Lambda}\vert \varphi(v(x))\vert^2\leq (l+1)\sum_{i=0}^l\vert a_i\vert^2\int_{\Lambda}\vert v(x)\vert^{2i}dx\nonumber\\
&\leq&(l+1)\vert a_0\vert^2+(l+1)\vert a_1\vert\int_{\Lambda}\vert u(x)\vert^2dx\nonumber\\
&+&(l+1)\max_{2\leq i\leq l}\sup_{x\in\overline{\Lambda}}\vert v(x)\vert^{2i-2}\sum_{i=2}^l\vert a_i\vert^2\int_{\Lambda}\vert v(x)\vert^2dx\nonumber\\
&\leq&(l+1)\vert a_0\vert^2+(l+1)\vert a_1\vert^2\Vert v\Vert^2+(l+1)\max_{2\leq i\leq 2}\Vert v\Vert^{2i-2}_{\mathcal{C}}\left(\max_{2\leq i\leq l}\vert a_i\vert^2\right)\Vert v\Vert^2\nonumber\\
&\leq&(l+1)\vert a_0\vert^2+(l+1)\vert a_1\vert^2\Vert v\Vert^2+(l+1)\max_{2\leq i\leq 2}\left(\Vert v\Vert_{\mathcal{C}}+1\right)^{2i-2}\left(\max_{2\leq i\leq l}\vert a_i\vert^2\right)\Vert v\Vert^2\nonumber\\
&\leq&(l+1)\vert a_0\vert^2+(l+1)\vert a_1\vert^2\Vert v\Vert^2+(l+1)2^{2l-3}\left(\Vert v\Vert_{\mathcal{C}}^{2l-2}+1\right)\left(\max_{2\leq i\leq l}\vert a_i\vert^2\right)\Vert v\Vert^2\nonumber\\
&\leq& L_1+L_1\Vert v\Vert^2\left(1+\Vert u\Vert^{2l-2}_{\mathcal{C}}\right).
\end{eqnarray}
This completes the proof of \eqref{Polynome1}.   The proof of \eqref{Polynome2} is similar to that of \eqref{Polynome1} by using the following well known fact
\begin{eqnarray}
a^n-b^n=(a-b)\sum_{i=0}^{n-1}a^ib^{n-1-i},\quad a, b\in \mathbb{R},\quad n\geq 1.
\end{eqnarray}
\end{remark}
\begin{remark}
If in \rmref{Remarkpol} we take the constant term of $\varphi$ to be  $0$, then the hypothesis \eqref{Polysharp1} is fulfilled. 
\end{remark}




\end{document}